\documentclass[reqno,a4paper, 11pt]{amsart}

\usepackage[utf8]{inputenc}

\usepackage[a4paper=true,pdfpagelabels]{hyperref}
\usepackage{graphicx}

\usepackage{amsfonts,epsfig}
\usepackage{latexsym}
\usepackage{mathabx}
\usepackage{amsmath}
\usepackage{amssymb}

\usepackage{color}

\newtheorem{theorem}{Theorem}
\newtheorem{lemma}[theorem]{Lemma}

\newtheorem{proposition}[theorem]{Proposition}

\newtheorem{lettertheorem}{Theorem}
\newtheorem{letterlemma}[lettertheorem]{Lemma}

\theoremstyle{definition}

\theoremstyle{remark}

\numberwithin{equation}{section}

\newcommand{\set}[1]{\left\{#1\right\}}
\newcommand{\abs}[1]{\lvert#1\rvert}
\newcommand{\nm}[1]{\lVert#1\rVert}

\newcommand{\esc}[1]{\left < #1 \right >}
\newcommand{\B}{\mathcal{B}}
\newcommand{\D}{\mathbb{D}}
\newcommand{\DD}{\widehat{\mathcal{D}}}
\newcommand{\Dd}{\widecheck{\mathcal{D}}}
\newcommand{\M}{\mathcal{M}}
\newcommand{\DDD}{\mathcal{D}}

\newcommand{\N}{\mathbb{N}}
\newcommand{\Z}{\mathbb{Z}}
\newcommand{\R}{\mathbb{R}}

\newcommand{\C}{\mathbb{C}}

\newcommand{\ep}{\varepsilon}

\renewcommand{\phi}{\varphi}

\newcommand{\T}{\mathbb{T}}

\def\a{\alpha}       \def\b{\beta}        \def\g{\gamma}
           
\def\la{\lambda}     \def\om{\omega}      
       \def\t{\theta}       \def\f{\phi}
         \def\r{\rho}         \def\z{\zeta}
\def\F{\Phi}         \def\c{\chi}         
\def\G{\Gamma}

\def\omg{\widehat{\omega}}

\def\mug{\widehat{\mu}}

\renewcommand{\H}{\mathcal{H}}

\newenvironment{Prf}{\noindent{\emph{Proof of}}}
{\hfill$\Box$ }

\begin{document}

\title[Fractional derivative description of the Bloch space]{Fractional derivative description of the Bloch space}


\author{\'Alvaro Moreno}
\address{Departamento de An\'alisis Matem\'atico, Universidad de M\'alaga, Campus de
Teatinos, 29071 M\'alaga, Spain} 
\email{alvarommorenolopez@uma.es}

\author{Jos\'e \'Angel Pel\'aez}
\address{Departamento de An\'alisis Matem\'atico, Universidad de M\'alaga, Campus de
Teatinos, 29071 M\'alaga, Spain} \email{japelaez@uma.es}

\author{Elena de la Rosa}
\address{Departamento de An\'alisis Matem\'atico, Universidad de M\'alaga, Campus de
Teatinos, 29071 M\'alaga, Spain} 
\email{elena.rosa@uma.es}

\thanks{This research was supported in part by Ministerio de Ciencia e Innovaci\'on, Spain, project PID2022-136619NB-I00; La Junta de Andaluc{\'i}a,
project FQM210.}

\subjclass[26A33, 30H30]{26A33,30H30}

\begin{abstract}
We establish new characterizations of the Bloch space $\mathcal{B}$ which include
  descriptions  in terms of classical fractional derivatives.
 Being precise,
for an analytic function   $f(z)=\sum_{n=0}^\infty \widehat{f}(n) z^n$  in the unit disc $\mathbb{D}$,
 we define the fractional derivative 
 $
  D^{\mu}(f)(z)=\sum\limits_{n=0}^{\infty} \frac{\widehat{f}(n)}{\mu_{2n+1}} z^n
 $
induced by a radial weight $\mu$,
  where $\mu_{2n+1}=\int_0^1 r^{2n+1}\mu(r)\,dr$ are the odd moments of $\mu$.  
Then, we consider
 the space $
    \mathcal{B}^\mu$ of analytic functions $f$ in $\mathbb{D}$ such that $\nm{f}_{\mathcal{B}^\mu}=\sup_{z\in \D} \widehat{\mu}(z)\abs{D^\mu(f)(z)}<\infty$, 
where $\widehat{\mu}(z)=\int_{|z|}^1 \mu(s)\,ds$.

We prove that $\mathcal{B}^\mu$ is continously embedded in  $\B$ for any radial weight $\mu$, and  $\mathcal{B}=\mathcal{B}^\mu$ if and only if $\mu\in \mathcal{D}=\widehat{\mathcal{D}}\cap\widecheck{\mathcal{D}}$. 
A radial weight $\mu \in \widehat{\mathcal{D}}$ if $\sup_{0\le r <1}\frac{\widehat{\mu}(r)}{\widehat{\mu}\left(\frac{1+r}{2}\right)}<\infty$
and a radial weight $\mu \in \widecheck{\mathcal{D}}$ if there exist $K=K(\mu)>1$ such that $\inf_{0\le r<1}\frac{\widehat{\mu}(r)}{ \widehat{\mu}\left(1-\frac{1-r}{K}\right)}>1.$

\end{abstract}

\maketitle

\maketitle

\section{Introduction}
Let $\H(\D)$ denote the space of analytic functions in the unit disc $\D=\{z\in\C:|z|<1\}$. 
For a nonnegative function $\om\in L^1([0,1))$, the extension to $\D$, defined by 
$\om(z)=\om(|z|)$ for all $z\in\D$, is called a radial weight.
 For $0<p<\infty$ and such an $\omega$, the Lebesgue space $L^p_\om$ consists of complex-valued measurable functions $f$ on $\D$ such that
    $$
    \|f\|_{L^p_\omega}^p=\int_\D|f(z)|^p\omega(z)\,dA(z)<\infty,
    $$
where $dA(z)=\frac{dx\,dy}{\pi}$ is the normalized Lebesgue area measure on $\D$. The corresponding weighted Bergman space is $A^p_\om=L^p_\omega\cap\H(\D)$. Throughout this paper we assume $\widehat{\om}(z)=\int_{|z|}^1\om(s)\,ds>0$ for all $z\in\D$, for otherwise $A^p_\om=\H(\D)$. As usual, for the standard weights 
$\omega(z)=(\alpha+1)(1-|z|^2)^\alpha$, $\alpha>-1$, we simply write $L^p_\alpha$ and $A^p_\alpha$ for the corresponding Lebesgue and Bergman spaces. In addition, we denote $L^p=L^p_0$ and $A^p=A^p_0$.

For  a radial weight $\mu$,
the  fractional derivative of $f(z)=\sum\limits_{n=0}^{\infty} \widehat{f}(n) z^n\in \H(\D)$  (induced by $\mu$) is
\begin{equation}
    \label{Dmu}
    D^{\mu}(f)(z)=\sum\limits_{n=0}^{\infty} \frac{\widehat{f}(n)}{\mu_{2n+1}} z^n, \; z \in \D.
\end{equation}
 Here and from now on, 
$\mu_{2n+1}$ are the odd moments of $\mu$, and in general we write $\mu_x=\int_0^1r^x\mu(r)\,dr$
 for a radial weight $\mu$ and  $x\ge0$. It is clear that $D^\mu(f)$ is a polynomial if $f$ is a polynomial and it follows from the inequality 
$$\mu_{2n+1}
 \geq 
\ep^{n+\frac{1}{2}}\mug(\sqrt{\ep}), \quad 0<\ep<1,\quad n\in\N,$$ that $D^\mu(f)\in \H(\D)$
for each $f\in\H(\D)$.   If $\mu$ is the standard weight $\mu(z)=\beta(1-|z|^2)^{\beta-1}$, $\b >0$, $D^\mu(f)$ is nothing but 
\begin{equation}\label{eq:deffracder2}
D^{\beta} (f)(z)=\frac{2}{\G(\beta+1)}\sum\limits_{n=1}^{\infty} \frac{\G (n+\beta+1)}{\G(n+1)} \widehat{f}(n) z^n,\quad z\in\D,
\end{equation}
which basically coincides with the fractional derivative of order $\b>0$  introduced by Hardy and Littlewood in \cite[p. 409]{HLMathZ}.
The differences between \eqref{eq:deffracder2} and \cite[(3.13)]{HLMathZ}  are in the multiplicative factor $\frac{2}{\G(\beta+1)}$ and the inessential factor $z^\beta$. 
See  \cite{PeralaAASF20,ZhuPacific94} for related definitions or reformulations of classical and generalized fractional derivative.

The fractional derivative $D^\mu$ was introduced in \cite{PeldelaRosa22} by the second and third authors to study Littlewood-Paley formulas for Bergman spaces, and in particular they proved  that 
\begin{equation}\label{eq:intro1}
    \|f\|_{A^p}^p\asymp\int_\D|D^{\mu}(f)(z)|^p\widehat{\mu}(z)^p\,dA(z),\quad f\in\H(\D), \quad \mu\in\DDD.
    \end{equation}
Let us recall that a radial weight $\mu \in \DD$ if there exists $C=C(\mu)>0$ such that $ \widehat{\mu}(r)\leq C \widehat{\mu}(\frac{1+r}{2}), \; 0\leq r <1$, and a radial weight $\mu \in \Dd$ if there exist $K=K(\mu)>1$ and $C=C(\mu )>1$ such that $\widehat{\mu}(r)\geq C \widehat{\mu}(1-\frac{1-r}{K})$, $0\leq r <1.$ 
We denote $\mathcal{D}=\DD\cap\Dd$.

In this paper, we are interested in obtaining a $p=\infty$ version  of \eqref{eq:intro1}, therefore for each radial weight $\mu$ we  
    consider the normed space 
    $$
    \B^\mu = \set{f\in \H(\D) : \nm{f}_{\B^\mu}=\sup_{z\in \D} \mug(z)\abs{D^\mu(f)(z)}<\infty}.
    $$
A first natural approach to this problem lead us to the following question: Which are the radial weights $\mu$ such that
  $ \B^\mu=H^\infty$? As usual $H^\infty$ denotes the space of bounded analytic functions in $\D$.  As it could be expected, our first result provides a stark negative  answer to  this question.

\begin{proposition}\label{th:hinftynobmu}
Let $\mu$ be a radial weight. Then, $H^\infty\neq \B^\mu$.
\end{proposition}
As for the proof  Proposition~\ref{th:hinftynobmu}  we show that the embedding
$H^\infty \subset \B^\mu$ implies $\mu\in\DD$, and then we draw on this property of the weight to construct a function $f\in  \B^\mu\setminus H^\infty$.

In view of Proposition~\ref{th:hinftynobmu}, the Littlewood-Paley formula
$$ \|f\|_{A^p}^p\asymp\int_\D |f^{(n)}(z)|^p (1-|z|)^{np}\,dA(z), \quad n\in\N, $$ and \eqref{eq:intro1}, it is also natural to study the relationship between $\B^\mu$ and the classical Bloch space of 
    $f\in \H(\D)$ such that  $\nm{f}_{\B}= |f(0)|+\sup_{z\in \D}(1-|z|)|f'(z)| <\infty$.
We will also give thought to the analogous question for the little versions of both spaces;
$\B_0$  the space of $f\in \H(\D)$ such that $\lim\limits_{\abs{z}\to 1^-}(1-\abs{z}^2)\abs{f'(z)} = 0$ and
$
\B_0^\mu = \set{f\in \H(\D) : \lim\limits_{\abs{z}\to 1^-}\mug(z)\abs{D^{\mu}(f)(z)} = 0}.
$

Our first main result is the following.
\begin{theorem}\label{th: contencion mu radial}
	Let $\mu$ be a radial weight. Then, $\B^\mu$ is continously embedded in $\B$, that is
\begin{equation}\label{eq: contencion mu radial}
  \|f \|_{\B} \lesssim \nm{f}_{\B^\mu}, \quad f\in \H(\D).
\end{equation}
Moreover,  $\B_0^\mu$ is continously embedded in $\B_0$.
\end{theorem} 
In order to prove
 the first part of Theorem~\ref{th: contencion mu radial}
we use an appropriate   integral representation of the fractional derivative $D^\mu$
and
the well-known identification $\B\simeq (A^1)^\star$ via the $A^2$-pairing \cite[Theorem~5.3]{Zhu}. Consequently, the proof
 boils down to showing that 
each $f\in \B^\mu$ induced a bounded linear functional 
$L_f(g)=\langle g,f\rangle_{A^2}=\lim_{r\to 1^-}\int_{\D} f(rz)\overline{g(z)}\, dA(z)$  on $A^1$ and $\| L_f\|\lesssim   \nm{f}_{\B^\mu}$. 
The proof of the second part of Theorem~\ref{th: contencion mu radial} is based on the fact that $f \in \B_0^\mu$ if and only if
$\lim_{r\to 1^-}\| f-f_r\|_{\B^\mu}=0$. Here and on the following $f_r(z)=f(rz)$, $0\le r<1$.

On the other hand, it is 
 obvious  that $\B= \B^\mu$ if $\mu=1$ and it is known that $f\in \B$ if and only if
$f^{(n)}(1-|z|)^n\in L^\infty$, $n\in\N$, \cite[Theorem~4]{Zhu}.  In addition, 
we recall that the Bloch space can be characterized in
 terms the multiplier transformation
$f^{[\beta]}(z)=\sum_{n=0}^\infty (n+1)^\beta\widehat{f}(n)z^n$, 
which may also be regarded as fractional derivative
of order $\beta>0$ \cite{ChoeRinHung}.

Our next main result shows that characterizations of the Bloch space in terms of classical fractional derivatives are  examples of  a general phenomenon rather than  particular cases, and moreover it provides a neat characterization of the radial weights $\mu$ such that $\B=\B^\mu$.

\begin{theorem}\label{th: reverse}
	Let $\mu$ be a radial weight. Then, the following conditions are equivalent:
\begin{itemize}
\item[\rm (i)] $\mu\in\DDD$;
\item[\rm(ii)]   $\B$ is continously embedded in $\B^\mu$, that is
\begin{equation}\label{Desig Bmu radial} 
  \|f \|_{\B^\mu} \lesssim \nm{f}_{\B}, \quad f\in \H(\D);
\end{equation}
\item[\rm(iii)]  $\B^\mu=\B$ and  $$  \|f \|_{\B^\mu} \asymp \nm{f}_{\B^\mu}, \quad f\in \H(\D);$$
\item[\rm(iv)]   $\B_0$ is continously embedded in $\B_0^\mu$, that is
\begin{equation}\label{Desig Bmu radialBcero} 
  \|f \|_{\B^\mu} \lesssim \nm{f}_{\B}, \quad f \in\B_0;
\end{equation}
\item[\rm(v)]  $\B_0^\mu=\B_0$ with equivalence of norms.
\end{itemize}
\end{theorem}
The proof of Theorem~\ref{th: reverse} is lengthy and involved and requires some preparatory  results. In fact, the proof of (i)$\Rightarrow$(ii) is based on  asymptotic estimates of the integral means of order one of $D^\mu(B_{\zeta})$, where $B_{\zeta}(z)=(1-\overline{\z}z)^{-2}$ is the Bergman reproducing kernel of $A^2$. It is also worth mentioning that we use
smooth properties of universal Ces\'aro basis of polynomials introduced by  Jevti\'c  and Pavlovi\'c  \cite{JevPac98} to obtain the aforementioned  asymptotic estimates .
 Reciprocally, in order to prove (ii)$\Rightarrow$(i) we use that  $\mathcal{D}=\DD\cap\Dd=\DD\cap \M$, where
$\mu\in\M$ if there exist constants $C=C(\mu)>1$ and $K=K(\mu)>1$ such that $\mu_{x}\ge C\mu_{Kx}$ for all $x\ge1$.
 Moreover, 
 we  prove 
 that for every $f\in \B$, there exists $g\in L^\infty$ such that $P_\mu(g)=f$ and $\nm{g}_{L^\infty} \lesssim \nm{f}_\B$, where $P_\mu$ is the Bergman projection from $L^2_\mu$ 
to $A^2_\mu$. Then,  \cite[Theorem 2]{PR19} ensures that  $\mu\in \M$. 
We get the condition $\mu\in\DD$ by testing \eqref{Desig Bmu radial} on monomials. 
The rest of the proof follows from Theorem~\ref{th: contencion mu radial} and standard properties of the space $\B_0^\mu$.

 The next few lines are dedicated to offer a brief insight to  classes of radial weights
$\DD$, $\Dd$, $\DDD$ and $\M$.
Doubling weights appear in a natural way in many questions on operator theory. For instance, the Bergman projection $P_\mu$, induced by a radial weight $\mu$, acts as a bounded operator from the space $L^\infty$ of bounded complex-valued functions to the Bloch space $\mathcal{B}$ if and only if $\mu\in\DD$, and $P_\mu: L^\infty\to\B$ is bounded and onto if and only if $\mu\in\DDD
=\DD\cap\Dd=\DD\cap \M$.  It is known that 
$\Dd\subset \M$, \cite[Proof of Theorem~3]{PR19} and 
$\Dd\subsetneq \M$
 \cite[ Proof of Theorem~3 and Proposition~14]{PR19}. Further, each standard radial weight obviously belongs to $\DDD$, while $\Dd\setminus\DDD$ contains exponential type weights such as
	$$
	\mu(r)=\exp \left(-\frac{\a}{(1-r^l)^{\b}} \right)\quad 0<\alpha,l,\beta<\infty.
	$$
The class of rapidly increasing weights, introduced in \cite{PelRat}, lies entirely within $\DD\setminus\DDD$, and a typical example of such a weight is 
	$$
	\mu(z)=\frac1{(1-|z|^2)\left(\log\frac{e}{1-|z|^2}\right)^{\alpha}},\quad 1<\alpha<\infty.
	$$
To this end we emphasize that the containment in $\DD$ or $\Dd$ does not require differentiability, continuity or strict positivity. In fact, weights in these classes may vanish on a relatively large part of each outer annulus $\{z:r\le|z|<1\}$ of $\D$. For basic properties of the aforementioned classes, concrete nontrivial examples and more, see \cite{PelSum14,PelRat,PR19} and the relevant references therein.

The rest of the paper is organized as follows: Section~\ref{s2} is devoted to proving the  preliminary results needed in the proofs of the  main  results of the paper. Theorems~\ref{th: contencion mu radial} and \ref{th: reverse} are proved in  Section~\ref{s3} and Section~\ref{s4} contains a proof of Proposition~\ref{th:hinftynobmu}.

Finally, we introduce the following notation that has already been used above in the introduction. The letter $C=C(\cdot)$ will denote an absolute constant whose value depends on the parameters indicated
in the parenthesis, and may change from one occurrence to another.
We will use the notation $a\lesssim b$ if there exists a constant
$C=C(\cdot)>0$ such that $a\le Cb$, and $a\gtrsim b$ is understood
in an analogous manner. In particular, if $a\lesssim b$ and
$a\gtrsim b$, then we write $a\asymp b$ and say that $a$ and $b$ are comparable.

\section{Preliminary results}\label{s2}
\subsection{Background on weights}

Throughout the paper we will employ different descriptions of the classes of radial weights $\DD$, $\Dd$
and $\M$. 
The next result gathers several characterizations of $\DD$ proved in \cite[Lemma~2.1]{PelSum14}. 

\begin{letterlemma}
\label{caract. pesos doblantes}
Let $\om$ be a radial weight. Then, the following statements are equivalent:
\begin{itemize}
    \item[(i)] $\om \in \DD$;
    \item[(ii)] There exist $C=C(\om)\geq 1$ and $\a_0=\a_0(\om)>0$ such that
    $$ \omg(s)\leq C \left(\frac{1-s}{1-t}\right)^{\a}\omg(t), \quad 0\leq s\leq t<1,$$
    for all $\a\geq \a_0$;
    \item[(iii)]
    $$ \om_x=\int_0^1 s^x \om (s) ds\asymp \omg\left(1-\frac{1}{x}\right),\quad x \in [1,\infty);$$
  \item[(iv)] There exists $C(\om)>0$ such that $\om_n\le C \om_{2n}$, for any $n\in \N$.
  \end{itemize}
\end{letterlemma}

We will use a couple of descriptions of the class $\Dd$ which are known for experts. We include a proof for the sake of completeness.
\begin{lemma}
\label{caract. D check}
Let $\om$ be a radial weight. Then, the
following statements are equivalent:
\begin{itemize}
\item[(i)] $\om	\in \Dd$;
\item[(ii)] There exist $C=C(\om)>0$ and $\b=\b(\om)>0$ such that
$$\omg(s)\leq C \left(\frac{1-s}{1-t}\right)^{\b}\omg(t), \quad 0\leq  t\leq s<1;$$
\item[(iii)] For each (or some) $\g >0$ there exists $C=C(\g, \om)>0$ such that
$$ \int_{0}^r \frac{ds}{\omg(s)^{\g}(1-s)} \leq \frac{C}{ \omg(r)^{\g}},\quad 0\leq r < 1.$$
\end{itemize}
\end{lemma}
\begin{proof}
The equivalence (i)$\Leftrightarrow$(ii) was proved in \cite[Lemma~B]{PeldelaRosa22}. Next, assume (ii) holds. Then,
$$ \int_{0}^r \frac{ds}{\omg(s)^{\g}(1-s)} \leq 
\frac{C^\gamma(1-r)^{\g\b}}{\omg(r)^{\g}}\int_{0}^r \frac{ds}{(1-s)^{1+\g\b}} 
\lesssim 
\frac{1}{ \omg(r)^{\g}},\quad 0\leq r < 1.$$
Now assume that (iii) holds. Then, if $0\le t\le r<1$
$$ \frac{1}{\omg(t)^\g}\log\frac{1-t}{1-r}\le\int_{t}^r \frac{ds}{\omg(s)^{\g}(1-s)}\le \int_{0}^r \frac{ds}{\omg(s)^{\g}(1-s)} \leq \frac{C}{ \omg(r)^{\g}}.$$
So if $r=1-\frac{1-t}{K}$ and $K>1$, we get 
$$ \omg(t)\ge \left( \frac{\log K}{C}\right)^{1/\g} \omg\left(1-\frac{1-t}{K} \right), \quad 0\le t<1.$$
Consequently, taking $K>e^C$, we get $\om\in\Dd$. This finishes the proof.

\end{proof}

\subsection{Estimates of the integral means  of fractional derivative of Bergman reproducing kernel.  }
 
For any radial weight $\om$, the norm convergence in $A^2_{\om}$ implies the uniform convergence on compact subsets of $\D$, and hence each point evaluation~$L_z$ is a bounded linear functional on~$A^2_\om$. Therefore there exist Bergman reproducing kernels $B^\om_z\in A^2_\om$ such that
	$$
	L_z(f)=f(z)=\langle f, B_z^{\om}\rangle_{A^2_\om}=\int_{\D}f(\z)\overline{B_z^{\om}(\z)}\om(\z)\,dA(\z),\quad f \in A^2_{\om}.
	$$

In this section we are going to obtain asymptotic estimates of  the  integral means of order one of
$D^\mu(B_z^{\om})$ for $\omega,\mu\in\DD$.
In order to get this result, we need establish some notation and previous results. 

Let $W(z)=\sum\limits_{k \in J}b_k z^k$ be a polynomial, where $J$ denote a finite subset of $\N$ and $f(z)=\sum_{k=0}^\infty a_k z^k\in\H(\D)$. The Hadamard product
$$(W\ast f)(z)=\sum_{k=0}^\infty a_k b_kz^k,\quad z\in\D,$$
is well defined. Furthermore, it is easy to observe that
    \begin{equation}
    (W\ast f)(e^{it})
    =\frac{1}{2\pi}\int_{-\pi}^\pi W(e^{i(t-\theta)})f(e^{i\t})\,d\t.
    \end{equation}
For a given $C^\infty$-function $\Phi:\mathbb{R}\to\C$ with compact support,
set
    $$
    A_{\Phi,m}=\max_{x\in\mathbb{R}}|\Phi(x)|+m\max_{x\in\mathbb{R}}|\Phi^{(m)}(x)|,\quad m\in\N\cup\{0\},
    $$
and define the polynomials
    \begin{equation}
    W_n^\Phi(z)=\sum_{k\in\mathbb
    Z}\Phi\left(\frac{k}{n}\right)z^{k},\quad n\in\N.
    \end{equation}

The Hardy space $H^p$ consists of $f\in\H(\D)$ for which
    $\|f\|_{H^p}=\sup_{0<r<1}M_p(r,f)<\infty$,
where
    $
    M_p(r,f)=\left (\frac{1}{2\pi}\int_0^{2\pi}
    |f(re^{i\theta})|^p\,d\theta\right )^{\frac{1}{p}},\, 0<p<\infty,
    $
and 
    $
    M_\infty(r,f)=\max_{0\le\theta\le2\pi}|f(re^{i\theta})|.
$
The next result can be  found in \cite[pp.~111--113]{Pabook}.

\begin{lettertheorem}
\label{Th Polinomios Cesaro}
Let $\Phi:\mathbb{R}\to\C$ be a compactly supported $C^{\infty}$-function. Then, for each $0<p<\infty$ and $m \in \N$ with $mp>1$, there exists a constant $C=C(p)>0$ such that 
$$\Vert W_n^{\Phi}*f\Vert_{H^p}\leq C A_{\Phi,m}\Vert f \Vert_{H^p}$$
for all $f \in H^p$ and $n \in \N$.
\end{lettertheorem}

A particular case of the previous construction is useful for our purposes. By following \cite[Section~2]{JevPac98} (see also \cite[Proposition~4]{PeldelaRosa22} ), let $\Psi:\mathbb{R}\to\mathbb{R}$ be a $C^\infty$-function such that $\Psi\equiv1$ on $(-\infty,1]$, $\Psi\equiv0$ on $[2,\infty)$ and $\Psi$ is decreasing and positive on $(1,2)$. Set $\psi(t)=\Psi\left(\frac{t}{2}\right)-\Psi(t)$ for all $t\in\mathbb{R}$. Let $V_{0}(z)=1+z$
and
    \begin{equation}\label{vn}
    V_{n}(z)=W^\psi_{2^{n-1}}(z)=\sum_{k=0}^\infty
    \psi\left(\frac{k}{2^{n-1}}\right)z^j=\sum_{k=2^{n-1}}^{2^{n+1}-1}
    \psi\left(\frac{k}{2^{n-1}}\right)z^j,\quad n\in\N.
    \end{equation}

these polynomials have the following properties (see
\cite[p. 175--177]{JevPac98}):
    \begin{equation}
    \begin{split}\label{propervn}
    &f(z)=\sum_{n=0}^\infty (V_n\ast f)(z),\quad f\in\H(\D),\\
    &\|V_n\ast f\|_{H^p}\le C\|f\|_{H^p},\quad f\in H^p,\quad 0<p<\infty,\\
    &\|V_n\|_{H^p}\asymp 2^{n(1-1/p)}, \quad 0< p<\infty.
    \end{split}
    \end{equation}

\begin{proposition}
\label{medias orden 1}
Let  $\om ,\, \mu \in \DD.$ Then,
$$ M_1(r, D^{\mu}(B^{\om}_{a}))\asymp 1+\int_0^{r\vert a\vert}\frac{1}{\omg(t)\mug(t)(1-t)}dt, \quad a \in \D,\,r\in [0,1).$$
\end{proposition}
\begin{proof}
Firstly, assume $\frac{1}{2}\le |a|, r<1$.
Bearing in mind \eqref{propervn}, 
\begin{equation}\label{eq:m1}
M_1(r, D^{\mu}(B^{\om}_a ))\leq  \sum\limits_{n=0}^{\infty} \|V_n* (D^{\mu}(B^{\om}_a ))_r\|_{H^1}, 
\end{equation}
On the one hand, by Lemma \ref{caract. pesos doblantes} and the formula
$B^\om_a(z)=\sum_{n=0}^\infty \frac{(\overline{a}z)^n}{2\omega_{2n+1}}$, it follows that
\begin{equation}\begin{split}\label{geq:im1}
\|V_0* (D^{\mu} (B^{\om }_a) )_r\|_{H^1}&= \left\|\frac{1}{2\om_1\mu_1}+\frac{\overline{a}rz}{2\om_3\mu_3}\right\|_{H^1} \leq \frac{1}{2\om_1\mu_1} + \frac{|a|r}{2\om_3\mu_3}
\\ & \lesssim \frac{1}{\mug(r|a|)\omg(r|a|)(1-r|a|)} \int_{\frac{4r|a|-1}{3}}^{r|a|}dt
\\& \asymp \int_{\frac{4r|a|-1}{3}}^{r|a|}\frac{1}{\mug(t)\omg(t)(1-t)}dt \leq \int_{0}^{r|a|}\frac{1}{\mug(t)\omg(t)(1-t)}dt, \quad r, |a|\geq \frac{1}{2}. 
\end{split}\end{equation}

Analogously, 
\begin{equation}\label{geq:im1n}
\|V_1* (D^{\mu} (B^{\om}_a) )_r\|_{H^1} \lesssim \int_{0}^{r|a|}\frac{1}{\mug(t)\omg(t)(1-t)}dt, \quad r, |a|\geq \frac{1}{2}.
\end{equation}
 Now, for each  $n \in \N\setminus\{1\}$
 let us consider
the  functions
 $$\varphi_{1,n}(x)=\frac{r^x}{\om_{2x+1}}\chi_{[2^{n-1}, 2^{n+1}-1]}(x), \quad \frac{1}{2}\le r<1,$$
 $$\varphi_{2,n}(x)=\frac{|a|^x}{\mu_{2x+1}}\chi_{[2^{n-1}, 2^{n+1}-1]}(x), \quad \frac{1}{2}\le |a|<1,$$
and choose $C^\infty$-functions $\Phi_{1,n}$ and $\Phi_{2,n}$ with compact support contained in $[2^{n-2},2^{n+2}]$ such that $\Phi_{1,n}=\varphi_{1,n}$ and $\Phi_{2,n}=\varphi_{2,n}$ in $[2^{n-1}, 2^{n+1}-1]$.
By following the proof of   \cite[(3.12)]{PeldelaRosa22}, there exist $C_1=C_1(\om)>0$ and $C_2=C_2(\mu)>0$ such that  
\begin{equation}
\label{geq:im2}
A_{\Phi_{1,n},2}\le C_1\frac{ r^{2^{n-1}}}{\om_{2^{n}}}\quad \text{ and }\quad A_{\Phi_{2,n},2}\le C_1\frac{ |a|^{2^{n-1}}}{\mu_{2^{n}}},\quad \frac{1}{2}\le r,|a|<1.
\end{equation}
Then, if $a=|a|e^{i\t}$
\begin{equation*}
\begin{split}
 V_n*(D^{\mu} (B^{\om}_a))_r(z)&=\frac{1}{2}\sum\limits_{k=2^{n-1}}^{2^{n+1}-1}\psi\left(\frac{k}{2^{n-1}}\right)\frac{r^k}{\om_{2k+1}}\frac{|a|^k}{\mu_{2k+1}}(ze^{-i\t})^k
 \\ &=\frac{1}{2}\sum\limits_{k=2^{n-1}}^{2^{n+1}-1}\psi\left(\frac{k}{2^{n-1}}\right)\Phi_{1,n}(k)\Phi_{2,n}(k)(ze^{-i\t})^k\\
 &=(W_1^{\Phi_{1,n}}*W_1^{\Phi_{2,n}}*V_n)(ze^{-i\t}).
\end{split}
\end{equation*}
Therefore, Theorem \ref{Th Polinomios Cesaro}, \eqref{geq:im2} and \eqref{propervn} implies that there is $C=C(\mu,\om)>0$ such that for each $n\in\N\setminus\{1\}$
\begin{equation*}
\begin{split}
\Vert V_n*(D^{\mu}(B^{\om}_a))_r\Vert_{H^1}\le C A_{\Phi_{1,n}, 2} A_{\Phi_{2,n}, 2}\Vert V_n\Vert_{H^1}\le C \frac{\left( r|a|\right)^{2^{n-1}}}{\om_{2^n}\mu_{2^{n}}}, \quad\frac{1}{2}\le r, |a|<1.
\end{split}
\end{equation*}
Then, arguing as in the proof of
 \cite[Lemma~6]{PeldelaRosa22} and using the hypotheses $\om , \mu \in \DD$ it follows that 
\begin{equation}
\begin{split}\label{geq:mi3}
\sum\limits_{n=2}^{\infty} \|V_n* (D^{\mu} (B^{\om}_a) )_r\|_{H^1} &\lesssim \left(\sum\limits_{n=2}^{\infty}\frac{\left( r|a|\right)^{2^{n-1}}}{\om_{2^n}\mu_{2^{n}}}\right)
\\ & \lesssim \int_{0}^{r|a|}\frac{1}{\mug(t)\omg(t)(1-t)}dt, \quad r, |a|\geq \frac{1}{2}.
\end{split}
\end{equation}
By joining \eqref{eq:m1}, \eqref{geq:im1}, \eqref{geq:im1n} and \eqref{geq:mi3},
\begin{equation}
\label{r,a>1/2}
 M_1(r, D^{\mu}(B^{\om}_{a}))\lesssim \int_0^{r\vert a\vert}\frac{1}{\omg(t)\mug(t)(1-t)}dt, \quad r, |a| \geq \frac{1}{2}.
\end{equation}
Moreover, if $|a|\leq \frac{1}{2}$ or $r\leq \frac{1}{2}$,
\begin{equation}
\label{r>1/2 o a >1/2}
 M_1(r, D^{\mu}(B^{\om}_{a}))\leq \sum\limits_{n=0}^{\infty}\frac{\left(\frac{1}{2}\right)^n}{2\mu_{2n+1}\om_{2n+1}}\lesssim 1.
\end{equation}
Thus \eqref{r,a>1/2} and \eqref{r>1/2 o a >1/2} yields
$$  M_1(r, D^{\mu}(B^{\om}_{a}))\lesssim 1+\int_0^{r\vert a\vert}\frac{1}{\omg(t)\mug(t)(1-t)}dt.$$
The reverse inequality follows from Hardy's inequality \cite[Section 3.6]{Duren} and  the proof of \cite[Lemma~6]{PeldelaRosa22}. This finishes the proof.
\end{proof}

\subsection{Previous results on $\B^\mu$ and $\B_0^\mu$. }
For a radial weight $\omega$
 the orthogonal projection from~$L^2_\om$ to~$A^2_\om$ is given by
    $$
    P_\om(f)(z) = \int_\D f(\zeta)\overline{B^\om_z(\zeta)}\om(\zeta)\,dA(\zeta),\quad z\in\D.
    $$
If $\om=1$,  we simply write $P_\omega=P$ and $B^\om_z(\zeta)=B_z(\zeta)=(1-\overline{z}\zeta)^{-2}$, $z.\z\in\D$.

We begin with a useful representation of the fractional derivative of 
$P_\om(f)$, $ f \in L^p_{\om}$, $1<p<\infty$. 
\begin{proposition}
\label{repr.}
Let $\mu$ be a radial weight and $\om\in \DD$. Then, 
\begin{equation}\begin{split}\label{representation}
D^{\mu}(f)(z) &=\int_{\D} f(\z)D^{\mu}(B^{\om}_{\z})(z)\om (\z)dA(\z), \quad f\in A^1_\omega, \quad z\in\D,
\end{split}\end{equation}
and
\begin{equation*}\begin{split}\label{representationPw} 
D^{\mu}P_{\om}(f)(z) &=\int_{\D} f(\z)D^{\mu}(B^{\om}_{\z})(z)\om (\z)dA(\z), \quad  f \in L^p_{\om},\quad
1<p<\infty,\quad z \in \D.
\end{split}\end{equation*}  
\end{proposition}
\begin{proof}
Fix $z\in\D$ and observe that
\begin{equation*}
	\begin{split}
		\int_{\D} f(\z)D^{\mu}(B^{\om}_{\z})(z)\om (\z)dA(\z) &= \int_{0}^1 2r\om(r)\frac{1}{2\pi} \int_0^{2\pi} \left (\sum\limits_{n=0}^{\infty} \widehat{f}(n)r^n e^{in\t}\right ) \left ( \sum\limits_{k=0}^{\infty} \frac{z^k r^k e^{-ik\t}}{2\mu_{2n+1}\om_{2n+1}} \right ) d\t dr \\
		&= \int_0^1 \om(r)\sum\limits_{n=0}^{\infty} \frac{\widehat{f}(n)z^n}{\mu_{2n+1}\om_{2n+1}}r^{2n+1} dr \\
		&= \sum\limits_{n=0}^\infty \frac{\widehat{f}(n)}{\mu_{2n+1}}z^n = D^\mu (f)(z), \quad f \in A^1_\omega.
	\end{split}
\end{equation*}
 Therefore if $f\in A^1_{\om}$, \eqref{representation} holds. 
Next, if $f \in L^p_{\om}$ then $P_{\om}(f)\in A^p_{\om}$ by \cite[Theorem~7]{PR19} . Therefore, two applications of \eqref{representation} and Fubini's Theorem yield
\begin{align*}
D^{\mu}P_{\om}(f)(z)&=\int_{\D} P_{\om}(f)(\z)D^{\mu}(B^{\om}_{\z})(z)\om (\z)dA(\z)
\\&=\lim\limits_{\r\to 1} \int_{\D} P_{\om}(f)(\r\z)D^{\mu}(B^{\om}_{\z})(z)\om (\z)dA(\z)
\\& =\lim\limits_{\r\to 1} \int_{\D} \left(\int_{\D} f(u)B^{\om}_u(\r\z)\om(u)dA(u)\right)D^{\mu}(B^{\om}_{\z})(z)\om (\z)dA(\z)
\\&= \lim\limits_{\r\to 1} \int_{\D} f(u)\om(u) \left(\int_{\D} B^{\om}_u(\r\z) D^{\mu}(B^{\om}_{\z})(z) \om (\z)dA(\z) \right) dA(u)
\\&= \lim\limits_{\r\to 1} \int_{\D} f(u)D^{\mu} (B^{\om}_u)(\r z)\om(u)dA(u),
\end{align*}
Finally, bearing in mind that $D^{\mu} (B^{\om}_u)(\r z)$ converges uniformly in $u\in\D$ to $D^{\mu} (B^{\om}_u)(z)$ as $\r\to 1^-$, it follows that
\begin{equation*}
\label{limite}
\lim\limits_{\r\to 1} \int_{\D} f(u)D^{\mu} (B^{\om}_u)(\r z)\om(u)dA(u)=\int_{\D} f(u)D^{\mu} (B^{\om}_u)(z)\om(u)dA(u).
\end{equation*}
This finishes the proof.

\end{proof}

Now, we provide some useful descriptions of $\B_0^\mu$ which can be proved by standard techniques, so we omit its proof.
\begin{proposition}\label{caract B0mu}
 Let $\mu$ be a radial weight. Then,
\begin{itemize}
\item[\rm(i)] 
$\B^\mu_0$ is a closed subspace of $\B^\mu$.
\item[\rm(ii)]   Let $f\in \H(\D)$, then $f\in \B^\mu_0$ if and only if 
    $$
    \lim\limits_{r\to 1^-}\nm{f-f_r}_{\B^\mu} = 0,
    $$
    where $f_r(z)=f(rz)$, $0\le r<1$.
\item[\rm(iii)] $\B^\mu_0$ is the closure  in $\B^\mu$ of the set of polynomials.

\end{itemize}

\end{proposition}

\section{Main results.}\label{s3}

\begin{Prf}{\em{Theorem~\ref{th: contencion mu radial}.}}
In order to prove the first part of our statement we are going to see that 
    \begin{equation*}\label{eq:j10}
    \nm{f}_\B \lesssim \nm{zV_\mu(f)}_{L^\infty}, \quad f\in \H(\D),
    \end{equation*}
where  $V_\mu(f)(z) = \frac{\mug(z)}{\abs{z}}D^\mu(f)(z)$,  $z\in \D$. 
    Precisely, we will prove that the functional
    $$
    L_f(g) = \esc{g,f}_{A^2} = \lim_{r\to 1^-}\esc{g,f_r}_{L^2},
    $$
    belongs to $(A^1)^\ast$ and $\nm{L_f} \lesssim \nm{zV_\mu(f)}_{L^\infty}$, this fact together with \cite[Theorem~5.3]{Zhu} will finish the proof.
    
 Let $h\in A^1$, then by \eqref{representation}
    \begin{equation}
    \begin{split}\label{eq:j11}
    \esc{h, V_\mu(f_r)}_{L^2} &= \int_\D h(z)\frac{\mug(z)}{\abs{z}}\int_\D \overline{f_r(\z)}D^\mu(B_z)(\z) \,dA(\zeta)\,dA(z) \\
    &= \int_\D \overline{f_r(\z)}\left ( \int_\D h(z) D^\mu(B_z)(\zeta)\frac{\mug(z)}{\abs{z}}dA(z)\right ) dA(\zeta) 
\\ & = \esc{T(h),f_r}_{L^2}, \quad 0\le r<1,
    \end{split}
    \end{equation}
    where
    $$
    T(h)(z) = \int_\D h(\z)D^\mu(B_\z)(z)\frac{\mug(\z)}{\abs{\z}}dA(\z), \quad z\in \D.
    $$
Observe that   
 \begin{equation*}
        \begin{split}
            T(h)(z) &= \int_\D h(\z)D^\mu(B_\z)(z)\frac{\mug(\z)}{\abs{\z}}dA(\z) \\ 
            &= \int_0^1 2\mug(r)\frac{1}{2\pi}\int_0^{2\pi}\left ( \sum\limits_{n=0}^\infty \hat{h}(n)r^ne^{in\t}\right )\left ( \sum \limits_{k=0}^\infty \frac{z^kr^ke^{-ik\t}}{\mu_{2k+1}}(k+1) \right ) d\t dr \\
            &= \int_0^1 2\mug(r)\sum\limits_{n=0}^\infty\frac{n+1}{\mu_{2n+1}}\hat{h}(n)z^nr^{2n} dr \\
            &= \sum\limits_{n=0}^\infty \frac{2(n+1)\mug_{2n}}{\mu_{2n+1}}\hat{h}(n)z^n = \sum\limits_{n=0}^\infty \frac{2(n+1)}{2n+1}\hat{h}(n)z^n, \quad z\in \D.
        \end{split}
    \end{equation*}
So if $\tilde{g}(z) = \sum\limits_{n=0}^\infty \frac{2n+1}{2(n+1)}\hat{g}(n)z^n \in A^1$, 
 $g(z)=T (\tilde{g})(z)$ and by \eqref{eq:j11}
\begin{equation}\label{igualdad producto}
\esc{g,f_r}_{L^2} = \esc{\tilde{g},V_\mu(f_r)}_{L^2}, \quad 0\le r<1.
\end{equation}
Let us prove 
\begin{equation*}\label{eq:j12}
\nm{\tilde{g}}_{A^1} \lesssim \nm{g}_{A^1}, \quad g\in \H(\D).
\end{equation*}
Notice that $
\tilde{g}(z) = \sum\limits_{n=0}^\infty \frac{2n+1}{2(n+1)}\hat{g}(n)z^n = g(z) - \sum\limits_{n=0}^\infty \frac{1}{2(n+1)}\hat{g}(n)z^n = g(z)-h(z)$ 
and $2zh(z) = \sum\limits_{n=0}^{\infty} \frac{\hat{g}(n)}{n+1}z^{n+1}$ is the primitive of $g$ with value $0$ at the origin.  So, by the Littlewood-Paley equivalence  for $A^1$ \cite[Theorem 4.28]{Zhu} 
\begin{equation*}\label{norma g tilde}
        \nm{h}_{A^1} = \int_\D \abs{h(z)} dA(z) \asymp \int_\D \abs{2zh(z)}dA(z) \asymp \int_\D \abs{g(z)}(1-\abs{z})dA(z) \leq \nm{g}_{A^1},
\end{equation*}
and therefore $\nm{\tilde{g}}_{A^1} \leq \nm{g}_{A^1} + \nm{h}_{A^1}\lesssim  \nm{g}_{A^1} $. This fact together with
\eqref{igualdad producto} yields
\begin{equation}\label{funci desig}
    \begin{split}
        \abs{\esc{g,f_r}_{L^2}} &= \abs{\esc{\tilde{g},V_{\mu}(f_r)}_{L^2}} \leq \int_\D \abs{\tilde{g}(z)} \abs{V_\mu(f_r)(z)} dA(z) \\
        &= \int_\D \frac{\abs{\tilde{g}(z)}}{\abs{z}} \abs{zV_\mu(f_r)(z)} dA(z) \lesssim \nm{\tilde{g}}_{A^1} \nm{zV_\mu(f_r)}_{L^\infty} \lesssim \nm{g}_{A^1} \nm{zV_\mu(f_r)}_{L^\infty}.
    \end{split}
\end{equation}
Now, let us see that for every $0<r<1$
\begin{equation}\label{desigualdad norma V}
\nm{zV_\mu(f_r)}_{L^\infty} = \sup\limits_{z\in \D} \mug(z)\abs{D^\mu(f_r)(z)} \leq 
\| f\|_{\mathcal{B}^\mu}.
\end{equation}
In order to prove this last inequality, notice that $\mug(z)D^\mu(f_r)(z)$ is a continuous function in $\overline{\D}$ for every $0<r<1$, so there exists $z_r\in \overline{\D}$ where the supremum is reached. In addition,  $z_r\in \D$ because $\mug=0$ on $\T$. So,
\begin{equation*}
    \begin{split}
        \sup\limits_{z\in \D} \mug(z)\abs{D^\mu(f_r)(z)} &= \mug(z_r)\abs{D^\mu(f_r)(z_r)} %= \mug(z_r)\abs{D^\mu(f)(rz_r)} 
\\ & \leq 
 \mug(z_r)\sup\limits_{\abs{u} = \abs{z_r}}\abs{D^\mu(f)(u)} = \sup\limits_{\abs{u} = \abs{z_r}}\mug(u)\abs{D^\mu(f)(u)} \leq \| f\|_{\mathcal{B}^\mu}.
    \end{split}
\end{equation*}
Therefore, joining \eqref{funci desig} and \eqref{desigualdad norma V},  $L_f\in (A^1)^\ast$ and $\nm{L_f} \lesssim \| f\|_{\mathcal{B}^\mu}$.

Now, let us prove that $\B_0^\mu\subset\B_0$.  
By Proposition~\ref{caract B0mu}
    $
    \lim\limits_{r\to 1^-}\nm{f-f_r}_{\B^\mu} = 0
    $ for each $f  \in\B_0^\mu$,
    which together with \eqref{eq: contencion mu radial}
implies
    $$
    \lim \limits_{r\to 1^-}\nm{f-f_r}_\B \lesssim \lim\limits_{r\to 1^-}\nm{f-f_r}_{\B^\mu} = 0, \quad\text{ for each $f  \in\B_0^\mu$}.
    $$
This finishes the proof.
\end{Prf}

\begin{Prf}{\em{Theorem~\ref{th: reverse}.}}

(i)$\Rightarrow$(ii). Let $f\in \B$. Since $\B\subset A^1$, \eqref{representation} yields 
$$ D^{\mu} (f)(z)=\int_{\D} f(\z) D^{\mu }(B_{\z})(z)dA(\z).$$
Moreover, since the  Bergman projection $P
: L^{\infty}\to \B$ is bounded and onto there exists $h \in L^{\infty}$ such that $P(h)=f$ and $\|f\|_{\B}\asymp \|h \|_{L^{\infty}}$. Then,
\begin{align*}
D^{\mu} (f)(z)&=\int_{\D} \left(\int_{\D} h(u)B_{u}(\z)dA(u)\right)D^{\mu }(B_{\z})(z)dA(\z)
\\ &= \int_{\D} h(u) \left(\int_{\D}D^{\mu }(B_{\z})(z) B_{u}(\z)dA(\z)\right)dA(u)
\\&= \int_{\D} h(u)\overline{D^{\mu}(B_{z})(u)} dA(u).
\end{align*}
Hence, by using Proposition \ref{medias orden 1} and Lemma \ref{caract. D check}(iii)
\begin{align*}
|D^{\mu}(f)(z)| &\leq \|h\|_{L^{\infty}}\int_{\D}|D^{\mu}(B_{z})(u)| dA(u) \asymp  \|f\|_{\B} \int_0^1  \left(1+\int_0^{s|z|}\frac{1}{\mug(t)(1-t)^2} dt \right) ds 
\\&\leq  
\| f\|_{\B}\left(
1+ \int_0^1 \frac{|z|}{\mug(s|z|)(1-s|z|)} ds \right)\lesssim \frac{\| f\|_{\B}}{\mug(z)}, \quad z\in \D. 
\end{align*}
Therefore, (ii) holds.

(ii)$\Rightarrow$(i).  Firstly, let us prove that $\mu \in \DD$. Testing \eqref{Desig Bmu radial} on monomials we obtain 
 a constant $C=C(\mu)>0$ such that
    \begin{equation*}\label{testeo monomios}
        \mug(r)r^n \leq C\mu_{2n+1}, \quad n\in \N\cup\{0\}, \quad 0\le r<1.
\end{equation*}
Therefore
    \begin{equation*}\label{desig momentos}
    \begin{split}
        \mu_{\frac{3}{2}n+1} 
& = \left ( \frac{3}{2}n + 1\right )\int_0^1r^{\frac{3}{2}n}\mug(r) dr 
        \le  C \left ( \frac{3}{2}n + 1\right )\mu_{2n+1}\int_0^1r^{\frac{1}{2}n}dr 
\le 3C \mu_{2n+1}, \quad n\in \N\cup\{0\}.
    \end{split}
    \end{equation*}
    So, if $x\ge 1$ and $n\in \N $ are such that $n\leq x < n+1$, 
    \begin{equation*}
        \mu_{\frac{3}{2}x} \leq C \mu_{\frac{3}{2}x+1}\leq C\mu_{\frac{3}{2}n+1}\leq C\mu_{2n+1} \leq C\mu_{2x-2} \leq C \mu_{2x}.
    \end{equation*}
where $C>0$ is a constant that changes in every inequality only depending on $\mu$.
    Naming $y=\frac{3}{2}x$, 
    $$
    \mu_{y} \leq C \mu_{\frac{4}{3}y} \leq C^2 \mu_{\frac{16}{9}y} \leq C^3 \mu_{\frac{64}{27}y}\leq C^3\mu_{2y},\quad y\ge\frac{3}{2}.
    $$
    Consequently, $\mu\in \DD$  by Lemma~\ref{caract. pesos doblantes} (iv) .

   Now, let us prove  that $\mu\in \M$. By \cite[Theorem 2]{PR19} it is sufficient to prove that for every $f\in \B$, there exists $g\in L^\infty$ such that $P_\mu(g)=f$ and $\nm{g}_{L^\infty} \lesssim \nm{f}_\B$.
    Let $f(z)=\sum \limits_{n=0}^\infty \hat{f}(n)z^n\in \B$ and $h(z)=\sum \limits_{n=0}^\infty \hat{h}(n)z^n  \in \B^\mu$.
 Then, 
    \begin{equation*}
        \begin{split}
            P_\mu(\mug \cdot D_{\mu}(h))(z) &= \int_\D \mug(\z)D_\mu(h)(\zeta)\overline{B^\mu_z(\z)}\mu(\z) dA(\z) \\
            &= \int_\D \left ( \sum \limits_{n=0}^\infty \frac{\hat{h}(n)\z^n}{\mu_{2n+1}}\right )\left (\sum\limits_{k=0}^\infty \frac{\overline{\z}^kz^k}{2\mu_{2k+1}} \right )\mu(\z)\mug(\z)dA(\z) \\
            &= \int_0^1 r\mu(r)\mug(r)\frac{1}{2\pi} \int_0^{2\pi} \left ( \sum \limits_{n=0}^\infty \frac{\hat{h}(n)r^n e^{in\t}}{\mu_{2n+1}}\right )\left (\sum\limits_{k=0}^\infty \frac{r^ke^{-ik\t}z^k}{\mu_{2k+1}} \right ) d\t dr \\
            &= \int_0^1\mu(r)\mug(r)\sum\limits_{n=0}^\infty \frac{\hat{h}(n)r^{2n+1}}{\mu_{2n+1}^2}z^n dr \\
            &= \sum\limits_{n=0}^\infty \frac{(\mu \mug)_{2n+1}}{\mu_{2n+1}^2}\hat{h}(n) z^n,\quad z\in\D.
        \end{split}
    \end{equation*}
    So, if we take $g(z) = \mug(z)D^\mu(h)(z)$ with $h$ defined by
    $
    \hat{h}(n) = \frac{\mu_{2n+1}^2}{(\mu \mug)_{2n+1}}\hat{f}(n),
    $
all that we need to prove is that $\la_n=\frac{(\mu_{2n+1})^2}{(\mu \mug)_{2n+1}}$,
$n\in\N\cup\{0\}$, is a coefficient multiplier of $\B$. 
Then, 
   by Theorem~\ref{th: contencion mu radial}
$$   \nm{g}_{L^\infty}=\|h\|_{\B^\mu}\lesssim \nm{h} _{\B}\lesssim  \nm{f}_{\B}
    $$
and $P_\mu(g) = f$.
    In order to prove that $\set{\la_n}_{n=0}^\infty$ is a coefficient multiplier of $\B$ it is enough to see that
    \begin{equation}\label{condicion multipli}
        \sup\limits_{0<r<1}(1-r)M_1(r,\la^{[1]}) <\infty.
    \end{equation}
This follows from the equivalence
$\|f\|_{\B}\asymp \sup_{z\in \D}(1-|z|)^\beta||f^{[\beta]}(z)|$  \cite{ChoeRinHung}
for the multiplier transformation 
$f^{[\beta]}(z)=\sum_{n=0}^\infty (n+1)^\beta\widehat{f}(n)z^n,  \, \beta>0$,  and the inequalities
  \begin{equation*}
    \begin{split}
    |(\lambda\ast f)^{[2]}(r^2e^{it})|(1-r)^{2}
    &=\left|\frac{1}{2\pi}\int_{-\pi}^\pi \lambda^{[1]}(re^{i(t+\theta)}) f^{[1]}(re^{-i\t})\,d\t\right|(1-r)^{2}\\
    &\le M_\infty(r, f^{[1]})M_1(r,\lambda^{[1]})(1-r)^{2}
    \lesssim\|f\|_\B M_1(r,\lambda^{[1]})(1-r),
    \end{split}
    \end{equation*}
    where $\la(z)=\sum \limits_{n=0}^\infty \la_n z^n$.

By \eqref{propervn}, 
    \begin{equation}\label{media integral Vn}
    M_1(r,\lambda^{[1]}) = \nm{(\la^{[1]})_r}_{H^1} \leq C(\mu) + \sum\limits_{n=2}^\infty \nm{V_n\ast(\la^{[1]})_r}_{H^1}, \quad 0<r<1.
    \end{equation}
    For each $n\in \N \setminus \set{1}$ and $r\in \left [\frac{1}{2},1 \right )$ consider
    $$
    F_n(x) = \frac{\mu_{2x+1}^2}{(\mu\mug)_{2x+1}}r^x \c_{[2^{n-1},2^{n+1}]}(x), \quad x\in \R, 
    $$
    and $G_n(x)=(x+1)F_n(x)$.
On the other hand, by Lemma~\ref{caract. pesos doblantes}(iii)
    \begin{equation}\label{equiv momentos}
        (\mu\mug)_{2x+1} =\frac{1}{2}\mug \left (1-\frac{1}{2x+1}\right )^2 \asymp \mu_{2x+1}^2,  \quad x\ge 0.
    \end{equation}
   In addition,  for each radial weight $\nu$ there exists a constant $C=C(\nu)>0$ such that
    $$
    \int_0^1 s^x\left (\log{\frac{1}{s}} \right )^n \nu(s) ds \leq C \nu_x, \quad n\in \set{1,2},\quad  x\geq 2.
    $$
    So,  a direct calculation implies that there exists a constant $C=C(\mu)>0$ such that
    \begin{equation}\label{eq:j20}
    \abs{G_n''(x)} \leq C\abs{G_n(x)}, \quad n\in \N \setminus \set{1}, \quad r\in \left [\frac{1}{2},1 \right ), \quad x\geq 2.
    \end{equation}

    Therefore, \eqref{eq:j20} and \eqref{equiv momentos} yield
    \begin{equation*}
    \begin{split}
        A_{F_n,2} &= \max\limits_{x\in [2^{n-1},2^{n+1}]} \abs{G_n(x)} + \max\limits_{x\in [2^{n-1},2^{n+1}]} \abs{G_n''(x)} \lesssim  \max\limits_{x\in [2^{n-1},2^{n+1}]} \abs{G_n(x)} \\
        &\lesssim \max\limits_{x\in [2^{n-1},2^{n+1}]}(x+1)r^x \lesssim 2^nr^{2^{n-1}},\quad n\in \N \setminus \set{1}.
        \end{split}
    \end{equation*}
    For each $n\in \N \setminus \set{1}$, choose a $C^\infty-$function $\F_n$ with compact support contained in $[2^{n-2},2^{n+2}]$ such that $\F_n = G_n$ on $[2^{n-1},2^{n+1}]$ and
    \begin{equation}\label{Afi}
        A_{\F_n,2} = \max\limits_{x\in \R}\abs{\F_n(x)} + \max\limits_{x\in \R}\abs{\F_n''(x)}\lesssim 2^nr^{2^{n-1}}, \quad n\in \N \setminus \set{1}.
    \end{equation}
    Since
    $$
    W_1^{\F_n}(z) = \sum\limits_{k\in \Z}\F_n(k)z^k = \sum\limits_{k\in \Z\cap [2^{n-2},2^{n+2}]} \F_n(k)z^k,
    $$
    \eqref{vn} yields
    \begin{equation*}
    \begin{split}
        V_n\ast(\lambda^{[1]})_r(z)  &=\sum\limits_{k=2^{n-1}}^{2^{n+1}-1}\f\left ( \frac{k}{2^{n-1}} \right )(k+1)\frac{\mu_{2k+1}^2}{(\mu\mug)_{2k+1}}r^kz^k = \sum\limits_{k=2^{n-1}}^{2^{n+1}-1}\f\left ( \frac{k}{2^{n-1}} \right )\F_n(k)z^k \\
        &= \left ( W_1^{\F_n}\ast V_n \right )(z), \quad n\in \N \setminus \set{1}.
    \end{split}
    \end{equation*}
    This together with Theorem~\ref{Th Polinomios Cesaro}, \eqref{Afi} and \eqref{propervn},
%\eqref{eq:j20},
 implies
    \begin{equation*}
    \begin{split}
        \nm{V_n\ast(\lambda^{[1]})_r}_{H^1}  &= \nm{W_1^{\F_n}\ast V_n}_{H^1} \lesssim A_{\F_n,2}\nm{V_n}_{H^1} \lesssim 2^nr^{2^{n-1}}\nm{V_n}_{H^1} \\
        &\lesssim2^nr^{2^{n-1}}, \quad r\in \left [ \frac{1}{2},1 \right ), \quad n\in \N \setminus \set{1}.
    \end{split}
    \end{equation*}
    which combined with \eqref{media integral Vn} gives
    $$
    M_1(r,\la^{[1]}) \lesssim \sum\limits_{n=2}^\infty 2^nr^{2^{n-1}} \lesssim \frac{1}{1-r}, \quad r\in \left [ \frac{1}{2},1\right ).
    $$
    Therefore \eqref{condicion multipli} holds, and thus $\mu \in \M$. 

The equivalence between (ii) and (iii) follows from Theorem~\ref{th: contencion mu radial}.
Therefore, we have already proved (i)$\Leftrightarrow$(ii)$\Leftrightarrow$(iii).

Next, let us prove (i)$\Rightarrow$(iv). Take $f\in \B_0$, then 
$\lim_{r\to 1^-}\| f_r-f\|_{\B_0}=0$, which together with (ii) implies 
that $f\in \B^\mu$ and  $\lim_{r\to 1^-}\| f_r-f\|_{\B^\mu_0}=0$. So, $f\in \B^\mu_0$  by Proposition~\ref{caract B0mu}.

Now let us prove (iv)$\Rightarrow$(ii). If (iv) holds, then
$$ \| f_r\|_{\B^\mu}\lesssim    \| f_r\|_{\B}\le \|f\|_{\B}, \quad f\in \B, \quad 0\le r<1.$$
Moreover, for any $z\in\D$
$$\widehat{\mu}(z)|D^\mu(f)(z)|=\lim_{r\to 1^-} \widehat{\mu}(rz)|D^\mu(f)(rz)|
\le \liminf_{r\to 1^-}\| f_r\|_{\B^\mu},$$
so 
$$\| f\|_{\B^\mu}\le  \liminf_{r\to 1^-}\| f_r\|_{\B^\mu}\lesssim \|f\|_{\B}, \quad f\in \B.$$
That is (ii) holds.

The equivalence between (iv) and (v) follows from Theorem~\ref{th: contencion mu radial}.
 This finishes the proof.

\end{Prf}

\section{Relationship between $\B^\mu$ and $H^\infty$.}\label{s4}

   For each $x\in\mathbb{R}$, $E(x)$ denotes the integer such that $E(x)\le x<E(x)+1$.
The following technical lemma will be used in the proof of Proposition~\ref{th:hinftynobmu}.
It can be deduced from the proof of \cite[Lemma 8]{PelRathg} but we prove it for the convenience of the reader.
\begin{lemma}\label{serie lagunar}
    Let $\om\in \DD$ such that $\omg(0)=1$. 
For every $n\in \N\cup\set{0}$ denote by $r_n$ the smallest $r_n\in [0,1)$ such that
    $
    \omg(r_n) = \frac{1}{2^n},
    $
    and
    $
    M_n=E\left ( \frac{1}{1-r_n} \right ).
    $
Then, 
    $$
  1+  \sum\limits_{n=0}^\infty 2^n r^{M_n} \asymp \frac{1}{\omg(r)},\quad 0\le r<1.
    $$
\end{lemma}
\begin{proof}
We may assume that $r_1\le r <1$. Then,
there exists $N\in \N$ such that $r_N \leq r < r_{N+1}$, and 
    $$
    \sum\limits_{n=0}^\infty 2^nr^{M_n} \geq \sum\limits_{n=0}^\infty 2^nr_N^{M_n} \geq 2^Nr_N^{M_N} \asymp 2^N \asymp 2^{N+1} = \frac{1}{\omg(r_{N+1})} \geq \frac{1}{\omg(r)}.
    $$
In order to prove the reverse inequality. Firstly, we deal with $r=r_N$ and we split the series in two terms:
    $$
    \sum\limits_{n=0}^N 2^n r_N^{M_n} \leq r_N^{M_N}\sum\limits_{n=0}^N2^n \asymp 2^{N+1}\asymp\frac{1}{\omg(r_N)}.
    $$
    For the remaining  term,  by Lemma~\ref{caract. pesos doblantes} (ii) there exists $\alpha>0$ and $C>0$ such that
    $$
    \frac{1-r_n}{1-r_{n+j}} \geq C \left (\frac{\omg(r_n)}{\omg(r_{n+j})} \right )^\frac{1}{\alpha} = C2^{\frac{j}{\alpha}},\quad n,j\in \N\cup\set{0}.
    $$
    This, the inequality $$\log\frac{1}{x}\geq 1-x,\quad  0<x\leq 1,$$ and the definition of $r_n$ give
    \begin{equation*}
    \begin{split}
    \sum\limits_{n=N+1}^\infty 2^nr_N^{M_n} &= 2^N  \sum\limits_{j=1}^\infty 2^jr_N^{M_{N+j}} =  2^N  \sum\limits_{j=1}^\infty 2^j  e^{-M_{N+j}\log{\frac{1}{r_N}}}  \\
    &\leq 2^N\sum\limits_{j=1}^\infty 2^j  e^{-\frac{1-r_N}{1-r_{N+j}}} \leq 2^N\sum\limits_{j=1}^\infty2^je^{-C2^{\frac{j}{\alpha}}} \asymp 2^N = \frac{1}{\omg(r_N)}.
    \end{split}
    \end{equation*}
    For $r_1\le r<1$, take  $N\in \N$ with $r_N\leq r < r_{N+1}$, and then 
$$
    \sum\limits_{n=0}^\infty 2^nr^{M_n} \leq \sum\limits_{n=0}^\infty 2^n{r_{N+1}}^{M_n} \lesssim \frac{1}{\omg(r_{N+1})} = \frac{2}{\omg(r_{N})} \leq \frac{2}{\omg(r)}.
    $$
    This finishes the proof.
\end{proof}

\begin{Prf}{\em{Proposition~\ref{th:hinftynobmu}}.}
  If $H^\infty \not\subset \B^\mu$, then its obvious that $H^\infty \neq \B^\mu$. \\ 
   Next assume $H^\infty \subset\B^\mu$. Then, bearing in mind Theorem~\ref{th: contencion mu radial} we have that convergence in $\B^\mu$ implies uniform convergence on compact subsets of $\D$, and then standard arguments imply that $H^\infty$ is continuously embedded in $\B^\mu$, that is there exists a constant $C>0$ such that
    $$
    \nm{f}_{\B^\mu} \leq C \nm{f}_{H^\infty}, \quad f\in H^\infty.
    $$
    Then, by testing this inequality on monomials
and arguing as in the proof of (ii)$\Rightarrow$(i) of Theorem~\ref{th: reverse} it follows that $\mu\in\DD$.
  Next,  without loss of generality assume that $\widehat{\mu}(0)=1$
and consider the function 
    $$
    f(z) = \mu_1 +\sum_{n=0}^\infty\mu_{2M_n+1}2^nz^{M_n}\in \H(\D),
    $$
where $\{M_n\}$ are associated to $\mu$ via the statement of Lemma~\ref{serie lagunar}. 
Let us prove $f\in \B^\mu\setminus H^\infty$.
By Lemma~\ref{serie lagunar},
    \begin{equation*}
        \begin{split}
            \nm{f}_{\B^\mu} &= \sup\limits_{z\in \D} \mug(z)\abs{D^\mu(f)(z)}  
= \sup\limits_{z\in \D} \mug(z)\left |1+\sum\limits_{n=0}^\infty 2^nz^{M_n}\right | 
\\ & \lesssim \sup\limits_{0\leq r < 1}\mug(r) \left( 1+\sum\limits_{n=0}^\infty 2^n r^{M_n} \right)
< \infty,
        \end{split}
    \end{equation*}
that is $f\in \B^\mu$.
On the other hand, by Lemma~\ref{serie lagunar} and Lemma~\ref{caract. pesos doblantes}(ii)
    \begin{equation*}
        \begin{split}
         \mu_1+  \sum\limits_{n=0}^\infty \mu_{2M_n+1}2^n &= 
\int_0^1 t\mu(t)   
\left( 1+
\sum\limits_{n=0}^\infty 2^n t^{2M_n}\right) dt 
\\ & \asymp\int_{0}^1t\frac{\mu(t)}{\mug(t^2)}\,dt 
 \asymp\int_{0}^1t\frac{\mu(t)}{\mug(t)}\,dt 
= \infty,
        \end{split}
    \end{equation*}
which implies that $ f\notin H^\infty$ because it has positive Taylor coefficients. This finishes the proof.
\end{Prf}

\end{document}